\documentclass{article}
\usepackage{amssymb}
\usepackage{amsfonts}
\usepackage{amsmath}

\newtheorem{theorem}{Theorem}

\newtheorem{conjecture}[theorem]{Conjecture}
\newtheorem{corollary}[theorem]{Corollary}

\newtheorem{definition}[theorem]{Definition}

\newtheorem{lemma}[theorem]{Lemma}

\newtheorem{remark}[theorem]{Remark}

\newenvironment{proof}[1][Proof]{\noindent\textbf{#1.} }{\ \rule{0.5em}{0.5em}}
\include {mak}
\begin{document}

\title{{\large UNIMODALITY OF ORDINARY MULTINOMIALS AND MAXIMAL
PROBABILITIES OF CONVOLUTION POWERS OF DISCRETE UNIFORM DISTRIBUTION}}
\author{Hac\`{e}ne Belbachir}
\maketitle

\begin{abstract}
\noindent We establish the unimodality and the asymptotic strong unimodality
of the ordinary multinomials and give their smallest mode leading to the
expression of the maximal probability of convolution powers of the discrete
uniform distribution. We conclude giving the generating functions of the
sequence of generalized ordinary multinomials and for\ an extension of the
sequence of maximal probabilities for convolution power of discrete uniform
distribution.
\end{abstract}

\noindent Since\renewcommand{\thefootnote}{} \footnotetext{%
\noindent\ \ \ \ \ \ \ \ \ \ \ \ \ \
\par
\noindent \textbf{Acknowledgement of support.} This research is partially
supported by LAID3 laboratory of USTHB and LAGA\ laboratory of University of
Paris 13.} the eighteenth century, the expression of the convolution power
of the discrete uniform distribution has been very well known (e.g. de
Moivre in 1711, see \cite[3rd ed., 1756]{moi} or \cite[1731]{moi1}). This
probability distribution arises in many practical situations including, in
particular, games with equal chance, random affectation of tasks for many
servers, and random walks. It is well known, see Dharmadhikari \& Joak-Dev
\cite[1988, p.\thinspace 108-109.]{dha}, that the convolution of two
discrete unimodal distributions may be non unimodal. However, if these
distributions are symmetric, we obtain a symmetric unimodal distribution. It
is a discrete analog of Wintner's Theorem \cite[1938]{win}. Knowing that the
convolution power of the discrete uniform distribution is symmetric
unimodal, the determination of the maximal probability (mode) of such a
distribution and its argument remains a question for consideration. As a
recent work on the problem, one can see the article by Mattner \& Roos \cite[%
2007]{mat} where they establish the upper bound for the maximal probability $%
c_{q,L}<\sqrt{6/\pi q\left( q+2\right) L}$ ($c_{q,L}$ being the maximal
probability of the $L$-th convolution power of the discrete uniform
distribution on $\left\{ 0,1,...,q\right\} $). Before them, there were
several works aiming at finding such a bound. For example, Siegmund-Schultze
\& von Weizs\"{a}cker \cite[2007]{sie} proved the existence of a constant $A$
such that $c_{q,L}<A/\left( q+1\right) \sqrt{L}$ and gave an application of
these upper bounds in the construction of a polygonal recurrence of a
two-dimensional random walk. Alternatively, our aim is to give an explicit
expression of the mode of the $L$-th convolution powers of the discrete
uniform distribution (section two) by means of the unimodality of the
ordinary multinomials for which we study also the strong unimodality
(section one), we end the paper (section three) by giving the generating
functions for the two sequences of\ generalized ordinary multinomials: $\{%
\binom{z}{n}_{q}\}_{n}$ and $\{\binom{nz}{n}_{q}\}_{n},$ $z\in
%TCIMACRO{\U{2102} }%
%BeginExpansion
\mathbb{C}
%EndExpansion
,$ and thus of $\{c_{q,2n/q}\}_{n},$ when $q$ is even.

\section{Unimodality of ordinary multinomials}

\noindent The ordinary multinomials are a natural extension of binomial
coefficients (see \cite[2007]{bel2} for a recent overview on ordinary
multinomials). Letting $q,L\in
%TCIMACRO{\U{2115} }%
%BeginExpansion
\mathbb{N}
%EndExpansion
$, for an integer $k=0,1,\ldots ,qL,$ the ordinary multinomial $\binom{L}{k}%
_{q}$ is the coefficient of the $k$-th term of the following multinomial
expansion%
\begin{equation}
\left( 1+x+x^{2}+\cdots +x^{q}\right) ^{L}=\sum\limits_{k\geq 0}\binom{L}{k}%
_{q}x^{k}.  \label{1}
\end{equation}%
with $\binom{L}{k}_{1}=\binom{L}{k}$ (being the usual binomial coefficient)
and $\binom{L}{k}_{q}=0$ for $k>qL.$ Using the classical binomial
coefficient, one has%
\begin{equation}
\binom{L}{k}_{q}=\sum\limits_{j_{1}+j_{2}+\cdots +j_{q}=a}\binom{L}{j_{1}}%
\binom{j_{1}}{j_{2}}\cdots \binom{j_{q-1}}{j_{q}}.  \label{2}
\end{equation}%
Readily established properties are the symmetry relation
\begin{equation}
\binom{L}{k}_{q}=\binom{L}{qL-k}_{q}
\end{equation}%
and the recurrence relation
\begin{equation}
\binom{L}{k}_{q}=\sum\limits_{m=0}^{q}\binom{L-1}{k-m}_{q}.  \label{4}
\end{equation}

\noindent As an illustration of the latter recurrence relation, we give the
triangles of pentanomial and hexanomial coefficients which are just an
extension, well known in the combinatorial literature, of the standard
Pascal triangle.

\begin{center}
Table 1: Triangle of \textit{pentanomial} coefficients: $\binom{L}{k}_{4}$
\end{center}

\noindent $%
\begin{array}{cccccccccccccccc}
L\backslash k & \text{\textit{0}} & \text{\textit{1}} & \text{\textit{2}} &
\text{\textit{3}} & \text{\textit{4}} & \text{\textit{5}} & \text{\textit{6}}
& \text{\textit{7}} & \text{\textit{8}} & \text{9} & \text{\textit{10}} &
\text{\textit{11}} & \text{\textit{12}} & \text{\textit{13}} &  \\
\text{\textit{0}} & \text{\textbf{1}} &  &  &  &  &  &  &  &  &  &  &  &  &
&  \\
\text{\textit{1}} & \text{1} & \text{1} & \text{\textbf{1}} & \text{1} &
\text{1} &  &  &  &  &  &  &  &  &  &  \\
\text{\textit{2}} & \text{1} & \text{2} & \text{3} & \text{4} & \text{%
\textbf{5}} & \text{4} & \text{3} & \text{2} & \text{1} &  &  &  &  &  &  \\
\text{\textit{3}} & \text{1} & \text{3} & \text{6} & \text{10} & \text{15} &
\text{18} & \text{\textbf{19}} & \text{{\small 18+}} & \text{{\small 15+}} &
\text{{\small 10+}} & \text{{\small 6+}} & \text{{\small 3+}} & \text{1} &
&  \\
\text{\textit{4}} & \text{1} & \text{4} & \text{10} & \text{20} & \text{35}
& \text{52} & \text{68} & \text{80} & \text{\textbf{85}} & \text{80} & \text{%
68} & \text{{\small =52}} & \text{35} & \text{20} & \cdots \\
\text{\textit{5}} & \text{1} & \text{5} & \text{15} & \text{35} & \text{70}
& \text{121} & \text{185} & \text{255} & \text{320} & \text{365} & \text{%
\textbf{381}} & \text{365} & \text{320} & \text{255} & \cdots%
\end{array}%
$

\begin{center}
Table 2: Triangle of \textit{hexanomial} coefficients: $\binom{L}{k}_{5}$
\end{center}

\noindent $%
\begin{array}{ccccccccccccccccl}
L\backslash k & \text{\textit{0}} & \text{\textit{1}} & \text{\textit{2}} &
\text{\textit{3}} & \text{\textit{4}} & \text{\textit{5}} & \text{\textit{6}}
& \text{\textit{7}} & \text{\textit{8}} & \text{9} & \text{\textit{10}} &
\text{\textit{11}} & \text{\textit{12}} & \text{\textit{13}} & \text{\textit{%
14}} &  \\
\text{\textit{0}} & \text{\textbf{1}} &  &  &  &  &  &  &  &  &  &  &  &  &
&  &  \\
\text{\textit{1}} & \text{1} & \text{1} & \text{\textbf{1}} & \text{\textbf{1%
}} & \text{1} & \text{1} &  &  &  &  &  &  &  &  &  &  \\
\text{\textit{2}} & \text{1} & \text{2} & \text{3} & \text{4} & \text{5} &
\text{\textbf{6}} & \text{5} & \text{4} & \text{3} & \text{2} & \text{1} &
&  &  &  &  \\
\text{\textit{3}} & \text{1} & \text{3} & \text{6} & \text{10} & \text{15} &
\text{21} & \text{25} & \text{\textbf{27}} & \text{\textbf{27}} & \text{25}
& \text{21} & \text{15} & \text{10} & \text{6} & \text{3} & \text{1} \\
\text{\textit{4}} & \text{1} & \text{4} & \text{10} & \text{20} & \text{35}
& \text{56} & \text{80} & \text{104} & \text{125} & \text{140} & \text{%
\textbf{146}} & \text{140} & \text{125} & \text{104} & \text{80} & \cdots \\
\text{\textit{5}} & \text{1} & \text{5} & \text{15} & \text{35} & \text{70}
& \text{126} & \text{205} & \text{305} & \text{420} & \text{540} & \text{651}
& \text{735} & \text{\textbf{780}} & \text{\textbf{780}} & \text{735} &
\cdots%
\end{array}%
$

\noindent Let us investigate the unimodality of the sequence $\{\binom{L}{k}%
_{q}\}_{k=0}^{m}$. A finite sequence of real numbers $\{a_{k}\}_{k=0}^{m}$ ($%
m\geq 1$) is called \textit{unimodal} if there exists an integer $l\in
\{0,\dots ,m\}$ such that the subsequence $\{a_{k}\}_{k=0}^{l}$ increases,
while $\{a_{k}\}_{k=l}^{m}$ decreases. If $a_{0}\leq a_{1}\leq \dots \leq
a_{l_{0}-1}<a_{l_{0}}=\dots =a_{l_{1}}>a_{l_{1}+1}\geq \dots \geq a_{m}$
then the integers $l_{0},\dots ,l_{1}$ are the \textit{modes} of $%
\{a_{k}\}_{k=0}^{m}$. In the case where $l_{0}=l_{1}$ we talk about a
\textit{peak}, otherwise the set of the values of the mode is called \textit{%
plateau}. For positive non increasing and non decreasing sequences,
unimodality is implied by log-concavity. A sequence $\{a_{k}\}_{k=0}^{m}$ is
said to be \textit{logarithmically concave (}log-concave for short\textit{)}
or \textit{strongly} \textit{unimodal} if $a_{l}^{2}\geq a_{l-1}a_{l+1}$%
\textit{, }$1\leq l\leq m-1$. Also if the sequence is \textit{strictly
log-concave} (SLC for short), i.e. if the previous inequalities are strict,
then the sequences have at most two consecutive modes (a peak or a plateau).
For these notions, one can see Belbachir \& Bencherif \cite[2007]{bel4},
Belbachir \& al \cite[2007]{bel3}, Bertin \& Theodorescu \cite[1984]{ber},
Brenti \cite[1994]{bre}, Comtet \cite[1970]{com}, Dharmadhikari \& Joak-Dev
\cite[1988]{dha}, Keilson \& Gerber \cite[1971]{kei}, Medgyessy \cite[1972]%
{med}, Sagan \cite[2007]{sag}, Stanley \cite[1986]{sta1} and \cite[1989]{sta}
and Tanny \& Zuker \cite[1974]{tan}. In the following, $\left\lfloor
a\right\rfloor $ denotes the greatest integer in $a.$

\ \ \ \ \ \ \ \ \ \ \ \ \ \ \ \ \ \ \ \

\noindent The first main result of this article is the following.

\begin{theorem}
\label{dd}Let $q\geq 1$ and $L\geq 0$ be integers. Then the sequence $\{%
\binom{L}{k}_{q}\}_{k=0}^{qL}$ is unimodal and its smallest mode is given by
\begin{equation*}
k_{L}:=\arg \max_{k}\binom{L}{k}_{q}=\left\lfloor \left( qL+1\right)
/2\right\rfloor ,
\end{equation*}%
Furthermore, we have the following recurrence relation%
\begin{equation*}
\binom{L}{k_{L}}_{q}=\sum\limits_{i\in I_{q}}\binom{L-1}{k_{L-1}+i}_{q},
\end{equation*}%
where
\begin{equation*}
I_{q}=\left\{
\begin{array}{ll}
\left\{ -q/2,\ldots ,q/2\right\} & \text{if }q\text{ is even,} \\
\left\{ -\left( q+1\right) /2,\ldots ,\left( q-1\right) /2\right\} & \text{%
if }q\text{ and }L\text{ are odd,} \\
\left\{ -\left( q-1\right) /2,\ldots ,\left( q+1\right) /2\right\} & \text{%
otherwise.}%
\end{array}%
\right.
\end{equation*}
\end{theorem}

\begin{proof}
It suffices, for each one of the two cases: $q$\ odd and $q$\ even, to
proceed by induction over $L$\ using the recurrence relation (\ref{4}).
\end{proof}

\begin{remark}
For odd $qL$\ we have a plateau of two modes: $qL/2$\ and $qL/2+1$.\
Otherwise we have a peak:\textbf{\ }$\left( qL+1\right) /2$.
\end{remark}

\section{Determining the maximal probability for convolution powers of
discrete uniform distribution}

\noindent We are now able to achieve our purpose: the expression of the
maximal probability of the $L$-th convolution powers of the discrete uniform
distribution. Let $U_{q}$ be the random variable of the discrete uniform
distribution on $\left\{ 0,1,...,q\right\} $\ and let $U_{q}^{\star L}$ be
its $L$-th convolution powers:%
\begin{equation*}
U_{q}:=\frac{1}{q+1}\left( \delta _{0}+\delta _{1}+\cdots +\delta
_{q}\right) \ \ \ \ \text{(}\delta _{a}\text{ is the Dirac measure).}
\end{equation*}

\noindent In \cite[2007]{bel2}, Belbachir and al. established a link between
the ordinary multinomials and the density probability of convolution powers
of discrete uniform distribution. With respect to the counting measure, such
a density is given by
\begin{equation*}
P\left( U_{q}^{\star L}=k\right) =\frac{\binom{L}{k}_{q}}{\left( q+1\right)
^{L}},\ \ k=0,1,\ldots ,qL.
\end{equation*}

\begin{remark}
From Odlyzko and Richmond \cite[1985]{odl} we know that for $L$\
sufficiently large, the sequence of probabilities $\{P\left( U_{q}^{\star
L}=k\right) \}_{k}$\ is strongly unimodal, from which we easily deduce that
the sequence $\{\binom{L}{k}_{q}\}_{k}$\ is also asymptotically strongly
unimodal.
\end{remark}

\begin{conjecture}
For each positive integer $q,$ the sequence $\{\binom{L}{k}_{q}\}_{k}$ is
SLC.
\end{conjecture}

\noindent From Theorem \ref{dd}, as second main result, we give the values
of $c_{q,L}:=\max_{k}\binom{L}{k}_{q}/\left( q+1\right) ^{L}.$

\begin{theorem}
The maximal probability of the $L^{th}$ convolution power of the discrete
uniform distribution over $\left\{ 0,1,\ldots ,q\right\} $ is%
\begin{equation*}
c_{q,L}=\frac{1}{\left( q+1\right) ^{L}}\binom{L}{\left\lfloor \left(
qL+1\right) /2\right\rfloor }_{q}.
\end{equation*}
\end{theorem}

\section{Some generating functions}

\noindent As a third main result, we give the generating functions for the
sequence of generalized ordinary multinomials, the sequences $\{\binom{z}{n}%
_{q}\}_{n}$ and $\{\binom{nz}{n}_{q}\}_{n},$ $z\in
%TCIMACRO{\U{2102} }%
%BeginExpansion
\mathbb{C}
%EndExpansion
,$ and the extended sequence of maximal probabilities for convolution power
of discrete uniform distribution: $\{c_{q,2n/q}\}_{n}$.

\begin{definition}
For $z\in
%TCIMACRO{\U{2102} }%
%BeginExpansion
\mathbb{C}
%EndExpansion
$, we define the generalized ordinary multinomials, as follows%
\begin{equation}
\binom{z}{k}_{q}:=\sum\limits_{k_{1}+k_{2}+\cdots +k_{q}=k}\frac{z\left(
z-1\right) \cdots \left( z-k_{1}+1\right) }{\left( k_{1}-k_{2}\right)
!\left( k_{2}-k_{3}\right) !\cdots \left( k_{q-1}-k_{q}\right) !k_{q}!}.
\end{equation}%
This definition is motivated by the relation (\ref{2}).
\end{definition}

\begin{lemma}
We have the following inequality%
\begin{equation*}
\sum\limits_{k_{1}+\cdots +k_{q}=k}\frac{z\left( z-1\right) \cdots \left(
z-k_{1}+1\right) }{\left( k_{1}-k_{2}\right) !\cdots \left(
k_{q-1}-k_{q}\right) !k_{q}!}=\sum\limits_{\substack{ h_{1}+2h_{2}+\cdots
+qh_{q}=k  \\ h_{1}+h_{2}+\cdots +h_{q}=k_{1}}}\frac{z\left( z-1\right)
\cdots \left( z-k_{1}+1\right) }{h_{1}!h_{2}!\cdots h_{q-1}!h_{q}!}.
\end{equation*}
\end{lemma}

\begin{theorem}
Let $z\in
%TCIMACRO{\U{2102} }%
%BeginExpansion
\mathbb{C}
%EndExpansion
$, the generating function for generalized ordinary multinomials is given by%
\begin{equation*}
\sum_{n\geq 0}\binom{z}{n}_{q}t^{n}=\left( 1+t+t^{2}+\cdots +t^{q}\right)
^{\alpha }.
\end{equation*}
\end{theorem}

\begin{proof}
Using the Lemma, we have $\sum_{n\geq 0}\binom{z}{n}_{q}t^{n}=\sum\limits
_{\substack{ h_{1}+2h_{2}+\cdots +qh_{q}=n  \\ h_{1}+h_{2}+\cdots +h_{q}=m}}%
\binom{z}{m}\frac{m!}{h_{1}!h_{2}!\cdots h_{q-1}!h_{q}!}t^{n}.$

On the other hand

$%
\begin{array}{ll}
\left( 1+t+t^{2}+\cdots +t^{q}\right) ^{z} & =\sum_{m\geq 0}\binom{z}{m}%
\left( t+t^{2}+\cdots +t^{q}\right) ^{m} \\
& \ \ \  \\
& =\sum_{m\geq 0}\binom{z}{m}\sum\limits_{\substack{ l_{1}+2l_{2}+\cdots
+ql_{q}=n  \\ l_{1}+l_{2}+\cdots +l_{q}=m}}\frac{m!}{l_{1}!l_{2}!\cdots
l_{q-1}!l_{q}!}t^{l_{1}+2l_{2}+\cdots +ql_{q}}.%
\end{array}%
$

We conclude by summation over $n\geq 0$ is equivalent to summation over $%
m\geq 0.$
\end{proof}

\begin{remark}
Problem 19 of Comtet \cite{com}, Vol.1, p. 172, states that%
\begin{equation*}
\sum_{n\geq 0}x^{n}\complement _{t^{n}}\left( 1+t+t^{2}\right) ^{n}=\left(
1-2x-3x^{2}\right) ^{-\frac{1}{2}},
\end{equation*}%
using the fact that the coefficient of$\ t^{n}$ in the development of $%
\left( 1+t+t^{2}\right) ^{n}:\complement _{t^{n}}\left( 1+t+t^{2}\right)
^{n}=\binom{n}{n}_{2}$ is $\max_{k}\binom{n}{k}_{2},$ we obtain the
following combinatorial identity%
\begin{equation*}
G_{2}\left( t\right) :=\sum_{n\geq 0}c_{2,n}t^{n}=\left( 1+\frac{t}{3}%
\right) ^{-1/2}\left( 1-t\right) ^{-1/2}.
\end{equation*}%
This last identity can be shown as the generating function of the sequence $%
\{c_{2,n}\}_{n}$.
\end{remark}

\begin{theorem}
Let $z\in
%TCIMACRO{\U{2102} }%
%BeginExpansion
\mathbb{C}
%EndExpansion
$, the generating function of the sequence $\{\binom{nz}{n}_{q}\}_{n}$ is
given by
\begin{equation*}
\sum_{n\geq 0}\binom{nz}{n}_{q}t^{n}=u\left( 1-z\frac{u+2u^{2}+\cdots +qu^{q}%
}{1+u+u^{2}+\cdots +u^{q}}\right) ^{-1},
\end{equation*}%
where $u$ is a solution of the equation \ $t=u\left( 1+u+u^{2}+\cdots
+u^{q}\right) ^{-z}.$
\end{theorem}

\begin{proof}
Use Hermite's Theorem \cite{com} for the function $t\mapsto t\left(
1+t+\cdots +t^{q}\right) ^{-z}.$
\end{proof}

\begin{theorem}
For $q$ even, the generating function of the sequence $\{c_{q,2n/q}\}_{n}$
is given by
\begin{equation*}
G_{q}\left( t\right) :=\sum_{n\geq 0}t^{n}c_{q,2n/q}=\left( 1-\frac{2}{q}%
\frac{u+2u^{2}+\cdots +qu^{q}}{1+u+u^{2}+\cdots +u^{q}}\right) ^{-1}=\frac{q%
}{2}\frac{1+\sum_{k=1}^{q/2}\left( u^{-k}+u^{k}\right) }{\sum_{k=1}^{q/2}k%
\left( u^{-k}-u^{k}\right) },
\end{equation*}%
where $u$ is a solution of the equation%
\begin{equation*}
t=u\left( \frac{q+1}{1+u+u^{2}+\cdots +u^{q}}\right) ^{2/q}=\left( \frac{q+1%
}{1+\sum_{k=1}^{q/2}\left( u^{-k}+u^{k}\right) }\right) ^{2/q}.
\end{equation*}
\end{theorem}

\begin{proof}
Use the above Theorem for $z=2/q,$ and the change of variable $t\rightarrow
\left( q+1\right) t.$
\end{proof}

\begin{remark}
The sequence $\{c_{q,2n/q}\}_{n}$ contains strictly the subsequence $%
\{c_{q,L}\}_{L}.$
\end{remark}

\begin{corollary}
For $q=4$, the generating function of $\{c_{4,n/2}\}_{n}$ is given for $t\in %
\left] -\sqrt{5},1\right[ $ by
\begin{equation*}
G_{4}\left( t\right) :=\sum_{n\geq 0}t^{n}c_{4,n/2}=\left( 1-\frac{1}{4}%
t^{2}-\frac{1}{8}t^{4}-\frac{1}{200}t\left( 5t^{2}+20\right) ^{\frac{3}{2}%
}\right) ^{-1/2}.
\end{equation*}
\end{corollary}

\begin{corollary}
We have the following identities%
\begin{equation*}
\sum_{n\geq 0}\left( -5\right) ^{-n}\binom{n/2}{n}_{4}=2\text{ \ and }%
\sum_{n\geq 0}\left( -1\right) ^{n}c_{4,\frac{n}{2}}=2/\sqrt{5}.
\end{equation*}
\end{corollary}

\begin{remark}
The generating function of the sequence $\{c_{4,n}\}_{n}$is given for $t\in %
\left] -1,1\right[ $ by
\begin{equation*}
\sum_{n\geq 0}t^{n}c_{4,n}=(G_{4}\left( \sqrt{\left\vert t\right\vert }%
\right) +G_{4}\left( -\sqrt{\left\vert t\right\vert }\right) )/2.
\end{equation*}
\end{remark}

\noindent \textbf{Acknowledgement} \ The author is grateful to Pr. A.
Aknouche for useful suggestions.

Hac\`{e}ne Belbachir

USTHB/ Facult\'{e} de Math\'{e}matiques.

BP 32, El Alia, 16111 Bab Ezzouar, Algiers, Algeria.

hbelbachir@usthb.dz and hacenebelbachir@gmail.com


\begin{thebibliography}{99}
\bibitem{bel2} Belbachir, H., Bouroubi, S. and Khelladi, A. (2007).
Connection between ordinary multinomials, generalized Fibonacci numbers,
Bell polynomials and convolution powers of discrete uniform distribution.
Preprint, http://arXiv.org/abs/math.CO 0708.2195v1.

\bibitem{bel3} Belbachir, H., Bencherif, F. and Szalay L. (2007).
Unimodality of certain sequences connected with binomial coefficients,
Journal of Integer Sequences, Vol. 10, Art. 07.2.3.

\bibitem{bel4} Belbachir, H., Bencherif, F. (2007). Unimodality of sequences
associated to Pell numbers, to appear in Ars Combinatoria.

\bibitem{ber} Bertin, E.\ M.\ J. and Theodorescu, R. (1984). Some
characterizations of discrete unimodality, Statist. Probab. Lett., \textbf{2}%
, 23--30.

\bibitem{bre} Brenti, F. (1994). Log-concave and unimodal sequences in
algebra, combinatorics, and geometry: an update. In Jerusalem combinatorics
'93, vol. 178 of Contemp. Math. AMS, Providence, RI, p. 71--89.

\bibitem{com} Comtet, L. (1970). Analyse combinatoire. Puf, Coll. Sup.
Paris, Vol. 1 \& Vol. 2.

\bibitem{dha} Dharmadhikari, S. and Joak-Dev, K. (1988). Unimodality,
Convexity and Applications. Academic Press, Boston, MA.

\bibitem{kei} Keilson, J. and Gerber, H. (1971). Some results for discrete
unimodality, J. Amer. Statist. Assoc., \textbf{66}, 386--389.

\bibitem{mat} Mattner, L. and Roos, B. (2007). Maximal probabilities of
convolution powers of discrete uniform distributions. Preprint,
http://arXiv.org/abs/math.PR 0706.0843v1.

\bibitem{med} Medgyessy, P. (1972). On the unimodality of discrete
distributions, Period. Math. Hungar., \textbf{2}, 245--257.

\bibitem{moi} de Moivre, A. (1967). The doctrine of chances. Third edition
1756 (first ed. 1718 and second ed. 1738), reprinted by Chelsea, N. Y.

\bibitem{moi1} de Moivre, A. (1731). Miscellanca Analytica de Scrichus et
Quadraturis. J. Tomson and J. Watts, London.

\bibitem{odl} Odlyzko, A. M. and Richmond, L.\ B. (1985). On the unimodality
of high convolutions, Ann. Probability, \textbf{13}, 299--306.

\bibitem{sag} Sagan, B.\ E. (2007). Composition inside a rectangle and
unimodality. Preprint, http://arXiv.org/abs/math.CO 0707.1052v1.

\bibitem{sie} Siegmund-Schultze, R. and von Weizs\"{a}cker, H. (2007). Level
crossing probabilities II: Polygonal recurrence of multidimensional random
walks. Adv. Math. \textbf{208}, 680--698.

\bibitem{sta} Stanley, R. P. (1989). Log-concave and unimodal sequences in
algebra, combinatorics, and geometry. In graph theory and its applications:
East and West (Jinan, 1986), vol. 576 of N. Y. Acad. Sci., p. 500--535.

\bibitem{sta1} Stanley, R. P. (1986). Enumerative combinatorics, Wadsworth
and Brooks / Cole, Monterey, California.

\bibitem{tan} Tanny, S., Zuker, M. (1974). On a unimodal sequence of
binomial coefficients, Discrete Math. \textbf{9}, 79-89.

\bibitem{win} Wintner, A. (1938). Asymptotic distributions and infinite
convolutions, Edwards Brothers, Ann Arbor, Michigan.
\end{thebibliography}
\end{document}